\documentclass{amsart}
\newtheorem{theorem}{Theorem}[section]
\newtheorem{lemma}[theorem]{Lemma}
\newtheorem{corollary}[theorem]{Corollary}
\newtheorem{proposition}[theorem]{Proposition}

\theoremstyle{remark}
\newtheorem{remark}[theorem]{\bf Remark}
\theoremstyle{definition}

\numberwithin{equation}{section}
\begin{document}

\title[A characterization of reflexive spaces of operators]{A characterization of reflexive spaces of operators}

\author{Janko Bra\v{c}i\v{c} and Lina Oliveira}

\address{Janko Bra\v{c}i\v{c}, University of Ljubljana, NTF, A\v{s}ker\v{c}eva c.\ 12, SI-1000 Ljubljana, Slovenia}
\email{janko.bracic@fmf.uni-lj.si}

\address{Lina Oliveira, Center for Mathematical Analysis, Geometry and Dynamical Systems {\em and}\\
Department  of Mathematics, Instituto Superior T\'ecnico, Universidade de Lisboa, Av. Rovisco Pais,\\ 1049-001 Lisboa, Portugal}
\email{linaoliv@math.tecnico.ulisboa.pt}
 
\keywords{Reflexive space of operators, order-preserving map}

\subjclass[2010]{Primary 47A15}

\thanks{The first author was supported by the Slovenian Research Agency through the research program P1-0222. 
The second author is partially funded by FCT/Portugal through the projects PEst-OE/EEI/LA0009/2013 and  EXCL/MAT-GEO/0222/2012.}

\begin{abstract}
	We show that for a linear space of operators ${\mathcal M}\subseteq {\mathcal B}(H_1,H_2)$ the following assertions are equivalent.
	(i) ${\mathcal M} $ is reflexive in the sense of Loginov--Shulman. (ii) There exists an order-preserving map $\Psi=(\psi_1,\psi_2)$ on a bilattice 
	$Bil({\mathcal M})$ of subspaces determined by ${\mathcal M}$, with $P\leq \psi_1(P,Q)$ and $Q\leq \psi_2(P,Q)$, for any pair $(P,Q)\in Bil({\mathcal M})$, and
	such that an operator $T\in {\mathcal B}(H_1,H_2)$ lies in ${\mathcal M}$ if and only if $\psi_2(P,Q)T\psi_1(P,Q)=0$ for all 
	$(P,Q)\in Bil( {\mathcal M})$. This extends to reflexive spaces the Erdos--Power type characterization of weakly closed bimodules over a nest algebra.
\end{abstract}

\maketitle

\section{Introduction and preliminaries} \label{sec01}
In  \cite{EP}, Erdos and Power  characterized the weakly closed bimodules of a nest algebra in terms of  
order homomorphisms on the  lattice  of invariant subspaces of the algebra. Deguang showed in \cite{Deg} 
that, given any reflexive subalgebra $\sigma$-weakly generated by its rank one operators, the $\sigma$-weakly 
closed bimodules over the algebra could analogously be characterized in terms of  order homomorphisms on the 
lattice of invariant subspaces of the algebra. Li and Li \cite[Proposition 2.6]{LL} have extended the mentioned 
results to the realm of Banach spaces. It is worth noticing that  the bimodules considered in the Erdos--Power theorems  
are implicitly reflexive subspaces in the sense of Loginov--Shulman (cf. \cite{LS}). 
The aim of the present paper is to extend this type of characterization to all such reflexive subspaces. 
The main result Theorem \ref{theo01} shows that, for every reflexive space ${\mathcal M}$ of operators between 
two complex Hilbert spaces, there exists an order homomorphism on a bilattice of subspaces determined by  ${\mathcal M}$
 which describes this subspace in the sense of Erdos--Power \cite[Theorem 1.5]{EP}. 

The proof of Theorem \ref{theo01} requires some auxiliary results appearing in Section 2. In the rest of the present 
section, apart from the notation,  we shall also establish the facts about bilattices needed in the sequel.
 
Let  $H$ be  a  complex Hilbert space, let ${\mathcal B}(H)$ be the Banach algebra of all bounded 
linear operators on $H$, and let ${\mathcal P}(H)$ be the set of all orthogonal projections on $H$. 
It is well known that ${\mathcal P}(H)$ is a lattice when endowed with the partial order relation defined, 
for all $P_1, P_2\in H$, by $P_1 \leq P_2\; \iff\; P_1 H \subseteq P_2H$. 
The join $P_1 \vee P_2$ is the orthogonal projection onto $\overline{P_1H+P_2 H}$ and the meet  
$P_1 \wedge P_2$ is the orthogonal projection onto $P_1H \cap P_2 H$. In fact, ${\mathcal P}(H)$  
is a complete lattice whose top and bottom elements are, respectively, the identity operator $I$ and the zero operator $0$.

Recall that the lattice $Lat({\mathcal U})$ of invariant subspaces of a subset ${\mathcal U}$ of ${\mathcal B}(H)$ is given by
\begin{equation*}
Lat({\mathcal U})=\{ P\in {\mathcal P}(H);\quad P^{\perp}TP=0,\quad \text{for all}\; T\in {\mathcal U}\},
\end{equation*}
where $P^{\perp}=I-P$. It is clear that $Lat({\mathcal U})$ is a sublattice of ${\mathcal P}(H)$ which is strongly closed 
and therefore complete, i.e., for every subset ${\mathcal F} \subseteq Lat({\mathcal U})$, the supremum $\vee {\mathcal F}$ and the 
infimum $\wedge {\mathcal F}$ lie in $Lat({\mathcal U})$ (see \cite{Hal}). 

If ${\mathcal U}\subseteq {\mathcal B}(H)$ is a non-empty subset, then let ${\mathcal U}^*=\{ T^*;\; T\in {\mathcal U}\}$. 
We say that ${\mathcal U}$ is \emph{selfadjoint} if ${\mathcal U}^*={\mathcal U}$. It is obvious that $P\in Lat({\mathcal U})$ if 
and only if $P^\perp\in Lat({\mathcal U}^*)$, i.e., $Lat({\mathcal U}^*)= Lat({\mathcal U})^{\perp}$. 

Let $H_1,H_2$ be complex Hilbert spaces. We endow the Cartesian product 
${\mathcal P}(H_1)$ $\times {\mathcal P}(H_2)$ with the partial order $\preceq$ which is defined, for all 
$(P_1,Q_1),$ $(P_2,Q_2)$ $\in {\mathcal P}(H_1)\times {\mathcal P}(H_2)$, by
\begin{equation} \label{eq04}
(P_1,Q_1)\preceq (P_2,Q_2)\qquad \text{if and only if}\qquad P_1\leq P_2\; \text{ and }\; Q_1\geq Q_2.
\end{equation}
Hence the  operations of join  and meet  are given, respectively, by
\begin{equation} \label{eq03}
\begin{split}
(P_1,Q_1)\vee (P_2,Q_2)&=(P_1\vee P_2,Q_1\wedge Q_2)\quad\text{and}\\
(P_1,Q_1)\wedge (P_2,Q_2)&=(P_1\wedge P_2,Q_1\vee Q_2).
\end{split}
\end{equation}
It follows that ${\mathcal P}(H_1)\times {\mathcal P}(H_2)$ together with $\preceq$ is a lattice as it contains all the binary joins 
and meets. From now on we write ${\mathcal P}(H_1)\times_{\preceq} {\mathcal P}(H_2)$ whenever we consider the Cartesian product to be  
endowed with the partial order \eqref{eq04}, i.e., with the lattice structure \eqref{eq03}. The corresponding notation will also be used for Cartesian products 
of subsets of ${\mathcal P}(H_1)\times {\mathcal P}(H_2)$. Unless otherwise stated, it is assumed that the partial order under consideration
will always be $\preceq$.

Following \cite{ST}, we call a subset ${\mathcal L}$ of ${\mathcal P}(H_1)\times_{\preceq} {\mathcal P}(H_2)$ a \emph{bilattice} if 
it is closed under the lattice operations \eqref{eq03} and contains the pairs $(0,0)$, $(0,I)$, and $(I,0)$. Examples of bilattices are 
${\mathcal P}(H_1)\times_{\preceq} {\mathcal P}(H_2)$, of course, and
\begin{equation} \label{eq08}
BIL({\mathcal U})=\{ (P,Q)\in {\mathcal P}(H_1)\times_{\preceq} {\mathcal P}(H_2);\quad QTP=0,\quad \text{for any}\; T\in {\mathcal U}\},
\end{equation}
where ${\mathcal U} \subseteq {\mathcal B}(H_1,H_2)$ is an arbitrary non-empty set.

Recall that, for a non-empty family ${\mathcal F} \subseteq {\mathcal P}(H)$,
\begin{equation*}
Alg({\mathcal F})=\{ T\in {\mathcal B}(H);\quad P^{\perp}TP=0,\quad \text{for all}\; P\in {\mathcal F}\}
\end{equation*}
is a weakly closed subalgebra of ${\mathcal B}(H)$ that contains the identity operator. A subalgebra 
${\mathcal A}$ of $ {\mathcal B}(H)$ is said to be \emph{reflexive} if $AlgLat({\mathcal A})={\mathcal A}$.
The notion of reflexive algebras has been generalized in several different directions; see \cite{Had} for a general view
of reflexivity and \cite{HIY} for a recently introduced  generalization. The concept of reflexivity is naturally extended to spaces of operators as follows.

For a non-empty family ${\mathcal F}\subseteq {\mathcal P}(H_1)\times_{\preceq} {\mathcal P}(H_2)$, let 
\begin{equation*}
Op({\mathcal F})=\{ T\in {\mathcal B}(H_1,H_2);\quad QTP=0,\quad \text{for all}\; (P,Q)\in {\mathcal F} \}.
\end{equation*}
It is easily seen that $Op({\mathcal F})$ is a weakly closed linear subspace of ${\mathcal B}(H_1,H_2)$.
A subspace ${\mathcal M} \subseteq {\mathcal B}(H_1,H_2)$ is said to be \emph{reflexive} if $Op BIL({\mathcal M})={\mathcal M}$. 
This definition is equivalent to that of Loginov and Shulman in \cite{LS}, where a subspace ${\mathcal M}$ is said to be reflexive if ${\mathcal M}$ 
coincides with its \emph{reflexive cover} 
\begin{equation*}
Ref({\mathcal M})=\{ S\in {\mathcal B}(H_1,H_2);\; Sx\in \overline{{\mathcal M} x},   \quad \text{for all}\; x\in H_1\}.
\end{equation*}
In fact,  $OpBIL({\mathcal M})=Ref({\mathcal M})$ (cf. \cite[p. 298]{ST}).

\begin{remark} \label{rem01}
Notice that,  if ${\mathcal A} \subseteq {\mathcal B}(H)$ is an algebra containing the identity operator, then  
$Ref({\mathcal A})=AlgLat({\mathcal A})$.
\end{remark}

\section{Subspaces and modules} \label{sec02}

For a linear subspace ${\mathcal M} \subseteq {\mathcal B}(H_1,H_2)$, let 
\begin{equation} \label{eq01}
{\mathcal A}_{{\mathcal M}}=\{ A\in {\mathcal B}(H_1);\quad TA\in {\mathcal M},\quad\text{for all}\; T\in {\mathcal M}\}
\end{equation}
and
\begin{equation} \label{eq02}
{\mathcal B}_{{\mathcal M}}=\{ B\in {\mathcal B}(H_2);\quad BT\in {\mathcal M},\quad\text{for all}\; T\in {\mathcal M}\}.
\end{equation}
It is easily seen that ${\mathcal A}_{{\mathcal M}}$ and $ {\mathcal B}_{{\mathcal M}}$ are algebras containing the identity operator 
and that ${\mathcal M}$ is ${\mathcal B}_{{\mathcal M}}$-${\mathcal A}_{{\mathcal M}}$-bimodule. It is clear that these are the largest  
subalgebras of ${\mathcal B}(H_1)$, respectively ${\mathcal B}(H_2)$, such that ${\mathcal M}$ is a bimodule over them. If 
${\mathcal M}$ is closed (respectively, weakly closed), then ${\mathcal B}_{{\mathcal M}}$ and ${\mathcal A}_{{\mathcal M}}$ are closed 
(respectively, weakly closed). Next we show that ${\mathcal A}_{{\mathcal M}}$ and ${\mathcal B}_{{\mathcal M}}$ are reflexive whenever
${\mathcal M}$ is a reflexive space.

\begin{proposition} \label{prop02}
If ${\mathcal M} \subseteq {\mathcal B}(H_1,H_2)$ is a reflexive space, then ${\mathcal A}_{{\mathcal M}}$ and 
$ {\mathcal B}_{{\mathcal M}}$ are reflexive algebras.
\end{proposition}

\begin{proof}
It will only be shown that ${\mathcal A}_{{\mathcal M}}$ is reflexive since the reflexivity of ${\mathcal B}_{{\mathcal M}}$ can be  
similarly proved. In view of Remark \ref{rem01}, it suffices to show that $Ref({\mathcal A}_{{\mathcal M}})={\mathcal A}_{{\mathcal M}}$. 
In other words, fixing $S\in Ref({\mathcal A}_{{\mathcal M}})$, we need to show that, for all $T\in {\mathcal M}$, the operator $TS$ lies in ${\mathcal M}$. 
Since this is trivially satisfied by $T=0$, henceforth we shall assume that $T\ne 0$.

Let $x\in H_1$ and $\varepsilon >0$ be arbitrary. Since $S\in Ref({\mathcal A}_{{\mathcal M}})$, there exists 
$A_{x,\varepsilon}\in {\mathcal A}_{{\mathcal M}}$ such that $ \| Sx-A_{x,\varepsilon}x\|<\varepsilon/\| T\|$. Hence 
$ \| TSx-TA_{x,\varepsilon}x\|\leq \| T\| \| Sx-A_{x,\varepsilon}x\|<\varepsilon $.
The operator $TA_{x,\varepsilon}$ lies in ${\mathcal M}$ and, therefore, we can conclude that $TS\in Ref({\mathcal M})={\mathcal M}$, as required.
\end{proof}

\begin{corollary} \label{cor01}
Let ${\mathcal M}$ be a linear subspace of  ${\mathcal B}(H_1,H_2)$. Then 
$Ref\bigl( {\mathcal A}_{{\mathcal M}}\bigr)\subseteq {\mathcal A}_{Ref({\mathcal M})}$
and $Ref\bigl( {\mathcal B}_{{\mathcal M}}\bigr)\subseteq {\mathcal B}_{Ref({\mathcal M})}$.
\end{corollary}

\begin{proof}
Let $A\in {\mathcal A}_{{\mathcal M}}$. If $T\in Ref({\mathcal M})$, then, for any $x\in H_1$ and any $\varepsilon>0$, there exists 
$S_{x,\varepsilon}\in {\mathcal M}$ such that $\| TAx-S_{x,\varepsilon}Ax\| <\varepsilon$. Since $S_{x,\varepsilon}A\in {\mathcal M}$,  we conclude 
that $TA\in Ref({\mathcal M})$. By Proposition \ref{prop02}, the algebra ${\mathcal A}_{Ref({\mathcal M})}$ is  reflexive,  from which follows 
that  $Ref\bigl( {\mathcal A}_{{\mathcal M}}\bigr)\subseteq Ref\bigl( {\mathcal A}_{Ref({\mathcal M})}\bigr)={\mathcal A}_{Ref({\mathcal M})}$.

The proof of the second inclusion is  similar.
\end{proof}

Let $tr(\cdot)$ be the trace functional and let ${\mathcal C}_1(H)\subseteq {\mathcal B}(H)$ be the ideal of trace-class operators. 
The dual of ${\mathcal C}_1(H)$ can be identified with ${\mathcal B}(H)$ by means of the pairing $\langle C,A\rangle=tr(CA^*)$, with  
$C\in {\mathcal C}_1(H),\; A\in {\mathcal B}(H)$. The preannihilator of a subset ${\mathcal U}\subseteq {\mathcal B}(H)$ is 
$ {\mathcal U}_{\perp}=\{ C\in {\mathcal C}_1(H);\; tr(CA^*)=0,\; \text{for all}\; A\in {\mathcal U}\}$ 
and the annihilator of ${\mathcal V}\subseteq {\mathcal C}_1(H)$ is 
${\mathcal V}^{\perp}=\{ A\in {\mathcal B}(H);\;tr(CA^*)=0,\; \text{for all}\;  C\in {\mathcal V}\}$. 
It is obvious that ${\mathcal U}_{\perp}$ and ${\mathcal V}^{\perp}$ are linear spaces and that a linear subspace 
${\mathcal M} \subseteq {\mathcal B}(H)$ is $\sigma$-weakly closed if and only if ${\mathcal M}=({\mathcal M}_{\perp})^{\perp}$.

If ${\mathcal U}, {\mathcal V}$ are two non-empty sets of operators, then we denote by ${\mathcal U}{\mathcal V}$ the set of all products $TS$, 
where $T\in {\mathcal U}$ and $S\in {\mathcal V}$.

\begin{proposition} \label{prop06}
Let ${\mathcal M}$ be a linear subspace of ${\mathcal B}(H)$. Then the following assertions hold.

\begin{itemize}
\item[(i)] $({\mathcal A}_{{\mathcal M}})^{*}={\mathcal B}_{{\mathcal M}^*}$.

\item[(ii)]  If ${\mathcal M}$ is $\sigma$-weakly closed, then ${\mathcal A}_{{\mathcal M}}=({\mathcal M}^{*}{\mathcal M}_{\perp})^{\perp}$ and 
${\mathcal B}_{{\mathcal M}}=({\mathcal M} {\mathcal M}_{\perp})^{\perp}$.

\item[(iii)]  If ${\mathcal M}$ is selfadjoint and $\sigma$-weakly closed, then ${\mathcal A}_{{\mathcal M}}={\mathcal B}_{{\mathcal M}}$  
is a $C^*$-algebra.

\item[(iv)]  If ${\mathcal M}$ is selfadjoint, $\sigma$-weakly closed, and reflexive, then 
${\mathcal A}_{{\mathcal M}}={\mathcal B}_{{\mathcal M}}$  is a von Neumann algebra.
\end{itemize}
\end{proposition}

\begin{proof}
(i) An operator $A\in {\mathcal B}(H)$ lies in $({\mathcal A}_{{\mathcal M}})^{*}$ if and only if $TA^*\in {\mathcal M}$ for every 
$T\in {\mathcal M}$. However this is equivalent to $AT^*\in {\mathcal M}^*$ for any $T^*\in {\mathcal M}^*$, which by the definition means that 
$A\in {\mathcal B}_{{\mathcal M}^*}$.

(ii) Let $A\in {\mathcal A}_{{\mathcal M}}$. For arbitrary $T\in {\mathcal M}$ and $C\in {\mathcal M}_{\perp}$, we have $tr(T^*CA^*)=tr(C(TA)^*)=0$, 
since $TA\in {\mathcal M}$. This proves that $A\in ({\mathcal M}^* {\mathcal M}_{\perp})^{\perp}$. On the other hand, if 
$A\in ({\mathcal M}^* {\mathcal M}_{\perp})^{\perp}$ and $T\in {\mathcal M}$, then $tr(C(TA)^*)=tr(T^*CA^*)=0$, for any $C\in {\mathcal M}_{\perp}$.  
Hence $TA\in ({\mathcal M}_{\perp})^{\perp}={\mathcal M}$. The second equality is similarly proved.

(iii) It follows from (ii) that ${\mathcal A}_{{\mathcal M}}={\mathcal B}_{{\mathcal M}}$. By (i), the algebra ${\mathcal A}_{{\mathcal M}}$ 
is selfadjoint and  closed since  ${\mathcal M}$ itself is closed. Hence ${\mathcal A}_{{\mathcal M}}$ is a $C^*$-algebra.

(iv) By (iii), ${\mathcal A}_{{\mathcal M}}={\mathcal B}_{{\mathcal M}}$ is a $C^*$-algebra. However, since ${\mathcal M}$ is reflexive, it is 
weakly closed and, therefore, ${\mathcal A}_{{\mathcal M}}$ is also weakly closed.
\end{proof}

\section{A characterization of reflexivity} \label{sec03}

Let ${\mathcal M}\subseteq {\mathcal B}(H_1,H_2)$ be a linear subspace and let ${\mathcal A}_{{\mathcal M}}$ and 
${\mathcal B}_{{\mathcal M}}$ be the algebras defined in \eqref{eq01}--\eqref{eq02}. The associated bilattice $BIL({\mathcal M})$ (see \eqref{eq08})
is very large. For our purposes it suffices to consider a smaller bilattice to be defined below. Firstly, we state the following lemma which is just a 
formalization of a remark in \cite[p. 298]{ST}. We include a short proof.

\begin{lemma} \label{lem01}
Let ${\mathcal M}$ be a linear subspace of  ${\mathcal B}(H_1,H_2)$. For any pair $(P,Q)\in BIL({\mathcal M})$, there exists a pair 
$(P^\prime,Q^\prime)\in BIL({\mathcal M})$ such that $P^{\prime}\in Lat({\mathcal A}_{{\mathcal M}})$, 
$Q^{\prime}$ $\in Lat({\mathcal B}_{{\mathcal M}})^{\perp}$, $P\leq P^\prime$, and $Q\leq Q^\prime$.
\end{lemma}

\begin{proof}
Let $P^\prime$ be the orthogonal projection onto $\overline{{\mathcal A}_{{\mathcal M}} PH_1}$ and let $Q^\prime$ be the orthogonal projection 
onto $\overline{{\mathcal B}_{{\mathcal M}}^{*} QH_2}$. It is obvious that $\overline{{\mathcal A}_{{\mathcal M}} PH_1}$ is 
invariant for any $A\in {\mathcal A}_{{\mathcal M}}$ and that $\overline{{\mathcal B}_{{\mathcal M}}^* QH_2}$ is invariant for any 
$B^*\in {\mathcal B}_{{\mathcal M}}^*$. Hence $P^\prime \in Lat({\mathcal A}_{{\mathcal M}})$ and 
$Q^\prime \in Lat({\mathcal B}_{{\mathcal M}}^*)=Lat({\mathcal B}_{{\mathcal M}})^{\perp}$. Observe that 
$PH_1 \subseteq \overline{{\mathcal A}_{{\mathcal M}} P H_1}$ and 
$QH_2 \subseteq \overline{{\mathcal B}_{{\mathcal M}}^* Q H_2}$, since both algebras contain the identity operator. 
Consequently, $P\leq P^\prime$ and $Q\leq Q^\prime$.

To prove that  $(P^\prime,Q^\prime)$ lies in  $BIL({\mathcal M})$, we have to see that, for any $T\in {\mathcal M}$,  the equality $Q^\prime TP^\prime=0$ holds, 
i.e., $TP^\prime H_1 \perp Q^\prime H_2$. Let $x\in H_1$ be arbitrary. For any $\varepsilon>0$, there exist 
$A_\varepsilon \in {\mathcal A}_{{\mathcal M}}$ and $x_\varepsilon \in H_1$ such that 
$\| P^\prime x-A_\varepsilon Px_\varepsilon\| <\varepsilon$, and therefore $\| TP^\prime x-TA_\varepsilon Px_\varepsilon\| <\varepsilon \| T\|$.
For arbitrary $B^*\in {\mathcal B}_{{\mathcal M}}^*$ and $y\in H_2$, we have 
$\langle TA_\varepsilon Px_\varepsilon,B^*Qy\rangle=\langle QBTA_\varepsilon Px_\varepsilon,y\rangle =0$, since $BTA_\varepsilon \in {\mathcal M}$. Hence
\begin{equation*}
\begin{split}
|\langle TP^\prime x,B^*Qy\rangle| &= |\langle TP^\prime x-TA_\varepsilon Px_\varepsilon,B^*Qy\rangle|\\
&\leq \| TP^\prime x-TA_\varepsilon Px_\varepsilon\|\| B^*Qy\| < \varepsilon \| T\| \| B^*Qy\|,
\end{split}
\end{equation*}
yielding  $TP^\prime x \perp {\mathcal B}^* QH_2$, from which it follows that $TP^\prime H_1 \perp Q^\prime H_2$.
\end{proof}

Let
\begin{equation*}
Bil({\mathcal M})=BIL({\mathcal M}) \cap \bigl(Lat({\mathcal A}_{{\mathcal M}})\times_{\preceq} Lat({\mathcal B}_{{\mathcal M}})^\perp\bigr).
\end{equation*}
It is clear that $Bil({\mathcal M})$ is a  bilattice. 

\begin{proposition} \label{prop09}
Let ${\mathcal M}$ be a linear subspace of  ${\mathcal B}(H_1,H_2)$. Then 
$$OpBIL({\mathcal M})=OpBil({\mathcal M}).$$
\end{proposition}

\begin{proof}
Since $Bil({\mathcal M})$ is a subset of $BIL({\mathcal M})$, it follows that $OpBIL({\mathcal M})$ $\subseteq OpBil({\mathcal M})$. 
To show that the reverse  inclusion also holds, we begin by fixing an operator $T\in OpBil({\mathcal M})$ and a pair of projections 
$(P,Q)\in BIL({\mathcal M})$. By Lemma \ref{lem01}, there exists a pair $(P^\prime,Q^\prime)\in Bil({\mathcal M})$ such that
$P\leq P^\prime$ and $Q\leq Q^\prime$. Hence $P^\prime P=P$ and $QQ^\prime=Q$. It follows that $QTP=QQ^\prime TP^\prime P=0$ and, therefore, $T$ lies in 
$OpBIL({\mathcal M})$, as required.
\end{proof}

Let ${\mathcal M} \subseteq {\mathcal B}(H_1,H_2)$ be a linear space.  Define 
$\phi:\,Lat({\mathcal A}_{{\mathcal M}})\to Lat({\mathcal B}_{{\mathcal M}})^\perp$ by
\begin{equation} \label{eq05}
\phi(P)=\vee \{ Q\in Lat({\mathcal B}_{{\mathcal M}})^\perp;\; (P,Q)\in Bil({\mathcal M})\},
\end{equation}
and similarly define $\theta:\,Lat({\mathcal B}_{{\mathcal M}})^\perp \to Lat({\mathcal A}_{{\mathcal M}})$ by
\begin{equation} \label{eq06}
\theta(Q)=\vee \{ P\in Lat({\mathcal A}_{{\mathcal M}});\; (P,Q)\in Bil({\mathcal M})\}.
\end{equation}
Observe that none of the sets appearing in \eqref{eq05}--\eqref{eq06} is empty as
$(P,0)$, $(0,Q)\in$ $Bil({\mathcal M})$, for any  $P\in Lat({\mathcal A}_{{\mathcal M}}), Q\in Lat({\mathcal B}_{{\mathcal M}})^\perp $. 
The next proposition lists some properties of the  maps $\phi$ and $\theta$.

\begin{proposition} \label{prop10}
Let ${\mathcal M}$ be a linear subspace of  ${\mathcal B}(H_1,H_2)$ and let $\phi, \theta$ be the maps defined in \eqref{eq05}--\eqref{eq06}. 
Then the following assertions hold.
\begin{itemize}
\item[(i)] $\phi$ and $\theta$ are order-reversing maps.

\item[(ii)]  $(P,\phi(P)), (\theta(Q),Q)\in Bil({\mathcal M})$, for any $P \in Lat({\mathcal A}_{{\mathcal M}})$ and $Q\in Lat({\mathcal B}_{{\mathcal M}})^{\perp}$.

\item[(iii)] If ${\mathcal C}\subseteq Lat({\mathcal A}_{{\mathcal M}})$ and ${\mathcal D}\subseteq Lat({\mathcal B}_{{\mathcal M}})^\perp$ are non-empty sets, then
$\phi(\vee {\mathcal C})=\wedge \phi({\mathcal C})$ and $\theta(\vee {\mathcal D})=\wedge \theta({\mathcal D})$.

\item[(iv)] $P \leq \theta\phi(P)$ and $Q \leq \phi\theta(Q)$, for all $P\in Lat({\mathcal A}_{{\mathcal M}})$ and $Q\in Lat({\mathcal B}_{{\mathcal M}})^\perp$.

\item[(v)] $\phi \theta\phi=\phi$ and $\theta \phi \theta=\theta$.
\end{itemize}
\end{proposition}

\begin{proof}
Assertions (i)--(iv) will only be  proved for the map $\phi$, since the corresponding assertions concerning the map $\theta$ can be similarly proved. For the 
same reason, only the first equality in (v) will be proved.

(i) If $P_1, P_2 \in Lat({\mathcal A}_{{\mathcal M}})$ are such that $P_1\leq P_2$, then $P_1P_2=P_1=P_2P_1$. Hence, if $Q$ is a projection in  
${\mathcal P}(H_2)$ with $(P_2,Q)\in Bil({\mathcal M})$, then, for every $T\in {\mathcal M}$, we have  $QTP_1=QTP_2P_1=0$, yielding  
$(P_1,Q)\in Bil({\mathcal M})$. It follows that
\begin{equation*}
\begin{split}
\phi(P_2)&=\vee \{ Q\in Lat({\mathcal B}_{{\mathcal M}})^\perp;\;(P_2,Q)\in Bil({\mathcal M})\}\\
&\leq \vee \{ Q\in Lat({\mathcal B}_{{\mathcal M}})^\perp;\;(P_1,Q)\in Bil({\mathcal M})\}=\phi(P_1).
\end{split}
\end{equation*} 

(ii) Let $P\in Lat({\mathcal A}_{{\mathcal M}})$. We have to show that $ \phi(P)TP=0$, for every $T\in {\mathcal M}$. Let $T\in {\mathcal M}$, 
$x\in H_1$, $y\in H_2$ be arbitrary, and let $Q\in {\mathcal P}(H_2)$ be a projection such that $(P,Q)\in Bil({\mathcal M})$. 
Then $ \langle TPx,Qy\rangle=\langle QTPx,y\rangle=0$, that is to say that $TPH_1 \perp QH_2$. Since $\phi(P)$ is the orthogonal projection 
onto the closed linear  span of all the  spaces $QH_2$, where $Q$ is an orthogonal projection in ${\mathcal P}(H_2)$ such that 
$(P,Q)\in Bil({\mathcal M})$, we conclude that $TPH_1 \perp \phi(P)H_2$, i.e., $\phi(P)TP=0$. 

(iii) Let ${\mathcal C} \subseteq Lat({\mathcal A}_{{\mathcal M}})$ be a non-empty set. Then, for all  $P\in {\mathcal C}$, $P\leq \vee {\mathcal C}$ and, 
since $Lat({\mathcal A}_{{\mathcal M}})$ is complete,  $\vee {\mathcal C} \in Lat({\mathcal A}_{{\mathcal M}})$. It follows that, for all 
$P\in {\mathcal C}$, $\phi(\vee {\mathcal C})\leq \phi(P)$, as $\phi$ is an order-reversing map. Therefore 
$\phi(\vee {\mathcal C})\leq \wedge \phi({\mathcal C})$.

To show that this  inequality can be reversed, we shall prove firstly that $(\vee {\mathcal C},$ $ \wedge \phi({\mathcal C}))$ $\in Bil({\mathcal M})$. Let 
$T\in {\mathcal M}$ be arbitrary. Then, for every $P\in {\mathcal C}$,
we have $\wedge \phi({\mathcal C})\leq \phi(P)$, from which it follows that $\bigl(\wedge \phi({\mathcal C})\bigr) \phi(P)=\wedge \phi({\mathcal C})$. Hence, 
for all $P\in {\mathcal C}$, 
\begin{equation*}
\bigl(\wedge \phi({\mathcal C})\bigr) TP=\bigl(\wedge \phi({\mathcal C})\bigr) \phi(P)TP=0
\end{equation*}  
and, consequently, 
$\bigl(\wedge \phi({\mathcal C})\bigr) T\bigl(\vee {\mathcal C}\bigr)=0$, i.e., 
$\bigl( \vee {\mathcal C}, \wedge \phi({\mathcal C})\bigr) \in Bil({\mathcal M})$. It follows, by the definition of 
$\phi(\vee {\mathcal C})$, that $\wedge \phi({\mathcal C}) \leq \phi(\vee {\mathcal C})$. 

(iv) Let $P\in Lat({\mathcal A}_{{\mathcal M}})$. By assertion (ii), we have $(P,\phi(P)),$ $(\theta(\phi(P)),$ $\phi(P))\in Bil({\mathcal M})$. Since, 
by the definition  \eqref{eq05}, the projection $\theta(\phi(P))$ is the largest $P^\prime\in Lat({\mathcal A}_{{\mathcal M}})$ such that 
$(P^\prime,\phi(P))\in Bil({\mathcal M})$, we conclude that $P \leq \theta\bigl(\phi(P)\bigr)$. 

(v) Let $P\in Lat({\mathcal A}_{{\mathcal M}})$ be arbitrary. By assertion (iv), we know that $\phi(P) \leq \phi\theta\phi(P)$. Moreover, since by (ii) 
of this proposition,  $(P,\phi(P))$ and $(\theta\phi(P),\phi\theta\phi(P))$ lie in the bilattice $Bil({\mathcal M})$, we have 
$\bigl(P\wedge \theta\phi(P),\phi(P)\vee \phi\theta\phi(P)\bigr)\in Bil({\mathcal M})$. Notice however that (iv) implies $P\wedge \theta\phi(P)=P$ and 
$\phi(P)\vee \phi\theta\phi(P)=\phi\theta\phi(P)$. Thus, $(P,\phi\theta\phi(P))\in Bil({\mathcal M})$. By  the definition of $\phi$, the projection $\phi(P)$ is 
the largest $Q\in Lat({\mathcal B}_{{\mathcal M}})^{\perp}$ having the property $(P,Q)\in Bil({\mathcal M})$. Hence, $\phi\theta\phi(P)\leq\phi(P)$. 
Consequently, for all  $P\in Lat({\mathcal A}_{{\mathcal M}})$, we have $\phi\theta\phi(P)=\phi(P)$. 
\end{proof}

Let $\Psi_1,\, \Psi_2:\;Bil({\mathcal M}) \to Bil({\mathcal M})$ be defined by
\begin{equation} \label{eq07}
\begin{split}
\Psi_1(P,Q)&=(\theta\phi(P),\phi(P))\qquad \text{and}\\ 
\Psi_2(P,Q)&=(\theta(Q),\phi\theta(Q))\qquad (P,Q)\in Bil({\mathcal M}).
\end{split}
\end{equation}
Observe that Proposition \ref{prop10}~(ii) guarantees that the maps $\Psi_1$ and $\Psi_2$ are well defined.
 
\begin{corollary} \label{cor02}
Let ${\mathcal M}$ be a linear subspace of ${\mathcal B}(H_1,H_2)$ and let 
$\Psi_1,\, \Psi_2:\;Bil({\mathcal M})$ $\to Bil({\mathcal M})$ be the maps defined in \eqref{eq07}. Then $\Psi_1$, $\Psi_2$ are order-preserving maps
and $\Psi_1(Bil({\mathcal M}))=\Psi_2(Bil({\mathcal M}))$.
\end{corollary}

\begin{proof} It easily follows from Proposition \ref{prop10}~(i) that  $\Psi_1$ and $\Psi_2$ are order-pre\-ser\-ving maps. That the images of $\Psi_1$ and $\Psi_2$ 
coincide is an immediate consequence of Proposition \ref{prop10}~(v).
\end{proof}

We are now able to prove our main result.

\begin{theorem} \label{theo01}
Let ${\mathcal M}$ be a linear subspace of $ {\mathcal B}(H_1,H_2)$ and let ${\mathcal A}_{{\mathcal M}}$, 
${\mathcal B}_{{\mathcal M}}$ be the algebras defined in \eqref{eq01}--\eqref{eq02}. The following assertions are equivalent.
\begin{itemize}
\item[(i)] ${\mathcal M}$ is a reflexive space.

\item[(ii)]  There exists a map $\Psi=(\psi_1,\psi_2): Bil({\mathcal M}) \to Bil({\mathcal M})$ such that $P\leq \psi_1(P,Q)$ and $Q\leq \psi_2(P,Q)$, for any 
pair $(P,Q)\in Bil({\mathcal M})$, and
$$
{\mathcal M}=\{T\in {\mathcal B}(H_1,H_2);\; \psi_2(P,Q)T\psi_1(P,Q)=0,\; \text{for all}\; (P,Q)\in Bil( {\mathcal M})\}.
$$

\item[(iii)]  There exists a map $\psi_1: Lat({\mathcal B}_{{\mathcal M}})^{\perp} \to Lat({\mathcal A}_{{\mathcal M}})$ such that $P\leq \psi_1(Q)$, 
for any pair $(P,Q)\in Bil({\mathcal M})$, and
$$
{\mathcal M}=\{T\in {\mathcal B}(H_1,H_2);\; QT\psi_1(Q)=0,\; \text{for all}\; Q\in Lat({\mathcal B}_{{\mathcal M}})^{\perp}\}.$$

\item[(iv)]  There exists a map $\psi_2: Lat({\mathcal A}_{{\mathcal M}}) \to Lat({\mathcal B}_{{\mathcal M}})^{\perp}$ such that $Q\leq \psi_2(P)$, 
for any pair $(P,Q)\in Bil({\mathcal M})$, and
$$
{\mathcal M}=\{T\in {\mathcal B}(H_1,H_2);\; \psi_2(P)TP=0,\; \text{for all}\; P\in Lat({\mathcal A}_{{\mathcal M}})\}.
$$
\end{itemize}
\end{theorem}

\begin{proof}
Firstly we show that (i)$\iff$(ii). Assume that ${\mathcal M}$ is a reflexive space.  Let $\Psi$ be the map $\Psi_1$ defined in \eqref{eq07}, and let 
${\mathcal F}=\Psi(Bil({\mathcal M}))$. Clearly, ${\mathcal F} \subseteq Bil({\mathcal M})$ and, therefore, 
$Op({\mathcal F}) \supseteq OpBil({\mathcal M})={\mathcal M}$. 

To reverse the inclusion, fix $T\in Op({\mathcal F})$. Observe that,  by Proposition \ref{prop10}~(iv),
for any pair $(P,Q)\in Bil({\mathcal M})$,  $P\leq \theta\phi(P)=\psi_1(P,Q)$ and, by the definition of the map $\phi$,  $Q\leq \phi(P)=\psi_2(P,Q)$. Hence, 
for all $(P,Q)\in Bil({\mathcal M})$, $P=\theta\phi(P)P$,  $Q=Q\phi(P)$ and, consequently,  
$$ QTP =Q\phi(P)T\theta\phi(P)P=0. $$
It follows that $T\in OpBil({\mathcal M})={\mathcal M}$, as required. 

Conversely, suppose that there exists a map $\Psi=(\psi_1, \psi_2)$ as  stated in (ii).
It has to be shown that ${\mathcal M}=Op Bil({\mathcal M})$. Since it is clear that ${\mathcal M}\subseteq Op Bil({\mathcal M})$, it remains to show that 
${\mathcal M}\supseteq Op Bil({\mathcal M})$. Let $S\in Op Bil({\mathcal M})$ be arbitrary. Hence, for any pair $(P',Q')\in Bil({\mathcal M})$, we have 
$Q'SP'=0$. In particular, since for $(P,Q)\in Bil({\mathcal M})$, the image 
$(\psi_1(P,Q),\psi_2(P,Q))$ lies also  in $Bil{\mathcal M}$, it follows that $\psi_2(P,Q)T\psi_1(P,Q)=0$.
Finally, this yields that $S$ lies in the set 
$$\{T\in {\mathcal B}(H_1,H_2);\; \psi_1(P,Q)T\psi_2(P,Q)=0\quad \forall\, (P,Q)\in Bil( {\mathcal M})\},$$
which  coincides with  ${\mathcal M}$, by the assumption.

The remaining equivalences are similarly proved.
Notice that to prove the implication (i)$\Rightarrow$(iii) (respectively, (i)$\Rightarrow$(iv)), we set $\psi_1=\theta$  (respectively, $\psi_2=\phi$).
\end{proof}

Observe that the maps appearing in the equalities characterizing a reflexive space ${\mathcal M}$ in Theorem \ref{theo01}  need not be 
unique (see \cite[Remark, p. 223]{EP}). In particular, the map $\Psi$ in Theorem \ref{theo01}~(ii) can be chosen to be order-preserving.


\end{document}